\documentclass[12pt, reqno]{amsart}
 \usepackage{amsmath, amsthm, amscd, amsfonts, amssymb, graphicx, color, float}
\usepackage[bookmarksnumbered, colorlinks, plainpages]{hyperref}

\newtheorem{theorem}{Theorem}[section]
\newtheorem{lemma}[theorem]{Lemma}
\newtheorem{proposition}[theorem]{Proposition}
\newtheorem{corollary}[theorem]{Corollary}
\theoremstyle{definition}
\newtheorem{definition}[theorem]{Definition}
\newtheorem{example}[theorem]{Example}

\theoremstyle{remark}

\numberwithin{equation}{section}

\begin{document}

\title[$*$-frames for operators on Hilbert modules]
{$*$-frames for operators on Hilbert modules}

\author[ Janfada]{Mohammad Janfada}
\address{Mohammad Janfada, Department of Pure Mathematics, Ferdowsi University of Mashhad, Mashhad,  P.O. Box 1159-91775, Iran}
\email{mjanfada@gmail.com}
\author[ Dastourian]{Bahram Dastourian}
\address{Bahram Dastourian, Department of Pure Mathematics, Ferdowsi University of Mashhad, Mashhad, P.O. Box 1159-91775, Iran}
\email{bdastorian@gmail.com}



\begin{abstract}
$K$-frames were introduced by L. G\u{a}vru\c{t}a to study atomic systems on Hilbert spaces. Recently some generalizations of this concept are introduced and some of its difference with ordinary frames are studied. In this paper $*$-$K$-frames  are introduced and some properties of this generalization of  $K$-frames are studied. After proving some characterizations of $*$-$K$-frames, direct sum and tensor product of $*$-$K$-frames are considered and finally some perturbation results are established.
\end{abstract}

\maketitle

\section{Introduction and Preliminaries}
Frames in Hilbert spaces were introduced by J. Duffin and A.C. Schaffer \cite{Duf} in 1952. Now frames play an important role not only in the theoretics but also in many kinds of applications, and have been widely applied in signal processing \cite{Fer}, sampling \cite{Eld1, Eld2}, coding and communications \cite{Stro}, filter bank theory \cite{Dud}, system modelling \cite{Bol}, and so on.\\
In contrast to frames there exist systems of functions generating proper subspaces even though they do not belong to them. These families were considered by H.G. Feichtinger and T. Werther in \cite{FW} and namely families of local atoms. Let $\{x_n\}_{n\in \mathbb N}$ be in the Hilbert space $\mathcal H$ such that there exists positive real number $\mu>0$ with $\sum_{n\in \mathbb N} |\langle x, x_n\rangle|^2\leq \mu\|x\|^2$ for all $x\in \mathcal H$ and let $\mathcal H_0$ be a closed subspace of $\mathcal H$ then the sequence $\{x_n\}_{n\in \mathbb N}$ is called a family of local atoms for $\mathcal H_0$ if there exists a sequence of linear functionals $\{c_n\}_{n\in \mathbb N}$ such that\\
$(i)$ $\exists C>0$ with $\sum_{n\in \mathbb N}|c_n(x)|^2\leq \nu\|x\|^2$,\\
$(ii)$ $x=\sum_{n \in \mathbb N}c_n(x)x_n$,\\
for all $x$ in $\mathcal H_0$.\\
The motivation of these systems is given by the problems arising in sampling theory \cite{PS}.\\
Atomic systems and $K$-frames, where $K$ is a bounded linear operator on separable Hilbert space $\mathcal H$, introduced by L. G\u{a}vru\c{t}a in \cite{Guv} as a generalization of family local atoms.
A sequence $\{ x_n\}_{n\in \mathbb{N}}$ in the Hilbert space $\mathcal H$ is called an atomic system for $K$ if the following statements hold\\
$(i)$ the series $\sum_{n\in \mathbb{N}}c_n x_n$ converges for all $c = (c_n) \in l^2$;\\
$(ii)$ there exists positive real number $\nu> 0$ such that for every $x\in \mathcal H$ there exists $a_x = (a_n)\in l^2$ such that $\|a_x\|_{l^2}\leq \nu\|x\|$ and $Kx =\sum_{n\in \mathbb{N}}a_n x_n$.\\
Also a sequence $\{ x_n\}_{n\in \mathbb{N}}$ is said to be a $K$-frame for $\mathcal H$ if there exists positive real numbers $\lambda, \mu$ such that
\[ \lambda\| K^*x\|^2\leq\sum_{n\in \mathbb{N}}|\langle x, x_n\rangle|^2\leq \mu\|x\|^2,~~ x\in \mathcal H.\]
  Frames are a special case of $K$-frame when $K$ is the identity operator. It is proved that a sequence $\{x_n\}_{n\in \mathbb N}$ is an atomic system for $K$ if and only if it is a $K$-frame \cite{Guv}. Also in this paper it is proved that a family of local atoms for $H_0$ is indeed a $P_{H_0}$-frame, where $P_{H_0}$ is the orthogonal projection on $H_0$. The concept of $K$-frames and its properties has recently been studied in \cite{AR, JD, XZG, XZSD}.\\
Let $\mathcal A$ be a unital C$^*$-algebra and let $\mathcal H$ be a left $\mathcal A$-module. $\mathcal H$ is a pre-Hilbert $\mathcal A$-module if $\mathcal H$ is equipped with an $\mathcal A$-valued inner product $\langle ., .\rangle:\mathcal H\times \mathcal H\rightarrow \mathcal A$ that possesses the following properties, \\
$(i)$~ $\langle f, f\rangle \geq 0$, for all $f\in \mathcal H$ and $\langle f, f\rangle=0$ if and only if $f=0$;\\
$(ii)$~ $\langle Af+Bg, h\rangle=A\langle f, h\rangle+B\langle g, h\rangle$, for all $A, B\in \mathcal A$ and $f, g, h\in \mathcal H $;\\
$(iii)$~ $\langle f, g\rangle=\langle g, f\rangle^*$, for all $f, g\in \mathcal H $;\\
$(iv)$~ $\langle \mu f, g\rangle=\mu\langle f, g\rangle$, for all $\mu\in \mathbb C$ and $f, g\in \mathcal H $;\\
The action of $\mathcal A$ on $\mathcal H $ is $\mathbb C$- and $\mathcal A$-linear, i.e., $\mu(Af)=(\mu A)f=A(\mu f)$, for every $\mu \in \mathbb C$, $a \in A$ and $f \in \mathcal H $. For $f \in \mathcal H $, we define $\|f\| = \|\langle f, f\rangle\|^{\frac{1}{2}}$. If $\mathcal H $ is complete with $\|.\|$, it is called a Hilbert $\mathcal A$-module or a Hilbert C$^*$-module over $\mathcal A$.\\
 The C$^*$-algebra $\mathcal A$ itself can be recognized as a Hilbert $\mathcal A$-module with the inner product $\langle A, B\rangle=AB^*$, for any $A, B\in \mathcal A$. For a C$^*$-algebra $\mathcal A$ the standard Hilbert $\mathcal A$-module $\ell^2(\mathcal A)$ is defined by
 \[
 \ell^2(\mathcal A)=\{ \{A_j\}_{j\in \mathbb N} : \sum_{j\in \mathbb N}A_jA_j^* \ \ \mbox{converges in}\ \ \mathcal A\}
 \]
with $\mathcal A$-inner product $\langle \{A_j\}_{j\in \mathbb N}, \{B_j\}_{j\in \mathbb N} \rangle=\sum_{j\in \mathbb N}A_jB_j^*$.
Let $\mathcal H$ and $\mathcal K$ be two Hilbert modules over C$^*$-algebra $\mathcal A$. A map $T:\mathcal H \rightarrow \mathcal K $ is said to be adjointable if there exists a mapping $T^*:\mathcal K \rightarrow \mathcal H $ satisfying $\langle Tf, g\rangle=\langle f, T^*g\rangle$ where $f\in \mathcal H $ and $g\in \mathcal K $. The mapping $T^*$ is called the adjoint of $T$. Note that an adjointable operator is bounded. The set of all operators from the $\mathcal A$-module $\mathcal H$ to the $\mathcal A$-module $\mathcal K $ admitting an adjoint is denoted by $Hom^*_{\mathcal A}(\mathcal H, \mathcal K)$. The algebra $Hom^*_{\mathcal A}(\mathcal H)=Hom^*_{\mathcal A}(\mathcal H, \mathcal H)$ is a $C^*$-algebra. Let $\mathcal H $ be a Hilbert module over a C$^*$-algebra $\mathcal A$ and $T\in Hom^*_{\mathcal A}(\mathcal H)$, then one have
\[ \langle T(f), T(f)\rangle\leq \|T\|^2\langle f, f\rangle,\]
for every $f\in \mathcal H $ \cite{L, WO}. \\
Let $\mathcal H$ be a Hilbert C$^*$-module and $\mathcal M\subseteq \mathcal H$ be a closed submodule of a Hilbert module $\mathcal H$ We define the orthogonal complement $\mathcal M^{\perp}$  of $\mathcal M$ by
\[ \mathcal M^{\perp}=\{g\in \mathcal H: \langle f, g\rangle=0, \quad \forall f\in \mathcal M\}.\]
Then $\mathcal M^{\perp}$ is also a closed submodule of the Hilbert module $\mathcal H$. However the equality $\mathcal H=\mathcal M\oplus \mathcal M^{\perp}$ is not fulfilled in general (see\cite{L}). If this holds for close submodule $\mathcal M\subseteq \mathcal H$ then we say that $\mathcal M$ is orthogonally complemented. \\
In our study we need the following generalization of the so-called Douglas theorem \cite{Doug} for Hilbert modules.
\begin{theorem}\label{adj} \cite{FY}
Suppose that $\mathcal H $, $\mathcal H _1$ and $\mathcal H_2$ are Hilbert modules over a C$^*$-algebra $\mathcal A$. If $T\in Hom^*_{\mathcal A}(\mathcal H_1, \mathcal H)$ and $S\in Hom^*_{\mathcal A}(\mathcal H_2, \mathcal H)$ with $\overline{R(S^*)}$  orthogonally complemented, then the following are equivalent:\\
$(i)$ $R(T)\subseteq R(S)$;\\
$(ii)$ $\mu TT^*\leq SS^*$, for some positive real number $\mu> 0$;\\
$(iii)$ There exists positive real number $\lambda> 0$ such that $\lambda \|T^*f\|^2\leq \|S^*f\|^2$, for all $f\in \mathcal H$;\\
$(iv)$ There exists an adjointable operator $Q:\mathcal H_1 \rightarrow \mathcal H _2$ such that $T=SQ$.
\end{theorem}
Another version of Douglas theorem for Hilbert modules which is proved in \cite{ZH} is as follows.
\begin{theorem}\label{adj2}
Suppose that $\mathcal H$ and $\mathcal H _1$ are Hilbert module over a C*-algebra $\mathcal A$. If $T\in Hom^*_{\mathcal A}(\mathcal H)$ and $S:\mathcal H _1\rightarrow \mathcal H$ is adjointable operator, $R(S)$ is closed, then the following are equivalent:\\
$(i)$ $R(T)\subseteq R(S)$;\\
$(ii)$ $\lambda TT^*f\leq SS^*f$, $f\in \mathcal H$, for some $\lambda>0$;\\
$(iii)$ There exists an adjointable operator $Q:\mathcal H \rightarrow \mathcal H _1$ such that $T=SQ$.
\end{theorem}
One can easily verify that each of the above conditions are also equivalent to the following condition,\\
$(iv)$ There exists positive real number $\mu>0$ such that $\|T^*f\|^2\leq \|S^*f\|^2$, $f\in \mathcal H$.\\
 Suppose that $\mathcal{A}$ and $\mathcal{B}$ two $C^*$-algebras, $\mathcal{H}$ is a Hilbert $\mathcal{A}$-module and also $\mathcal{K}$ is a Hilbert $\mathcal{B}$-module. Let $\mathcal{A}\otimes \mathcal{B}$ be the completion of $\mathcal{A}\otimes_{alg} \mathcal{B}$ with the spatial norm and the following operation and involution,
\[
(A\otimes B)(C\otimes D)=AC\otimes BD\quad , \quad (A\otimes B)^*=A^*\otimes B^*, \quad A\otimes B, C\otimes D\in \mathcal A\otimes \mathcal B.
\]
Then $\mathcal{A}\otimes \mathcal{B}$ is a $C^*$-algebra and for every $A \in \mathcal{A}$, $B\in \mathcal{B}$ we have $\|A \otimes B\| =\|A\|\|B\|$. The algebraic tensor product $\mathcal{H}\otimes_{alg} \mathcal{K}$  is a pre-Hilbert $\mathcal{A}\otimes \mathcal {B}$-module with the module action
\[(A \otimes B)(f \otimes g) = Af \otimes Bg~~~~~~\ \ (A \in \mathcal{A},~ B \in \mathcal{B}, f \in \mathcal{H}, g \in \mathcal{K}),\]
and $\mathcal{A} \otimes \mathcal{B}$-valued inner product
\[\langle f_1\otimes g_1, f_2 \otimes g_2\rangle = \langle f_1, f_2\rangle \otimes \langle g_1, g_2\rangle~~~\ \  (f_1, f_2 \in \mathcal{H},\ \ ~~ g_1, g_2 \in \mathcal{K}).\]
It is well-known that for $h = \sum_{i=1}^n f_i\otimes g_i$ in $\mathcal{H}\otimes_{alg} \mathcal{K}$ we have
\[\langle h, h\rangle = \sum_{i,j=1}^n\langle f_i, f_j\rangle \otimes \langle g_i, g_j\rangle\geq 0\]
and $\langle h, h\rangle = 0$ if and only if $h = 0$. The completion $\mathcal{H}\otimes\mathcal{K}$ of $\mathcal{H}\otimes_{alg}\mathcal {K}$, which is a Hilbert $\mathcal{A}\otimes\mathcal{B}$-module is called the tensor product of $\mathcal{H}$ and $\mathcal{K}$ (see \cite{L}). We note that if $A \in \mathcal{A}^+$ and $B \in\mathcal{B}^+$, then $A \otimes B\in (\mathcal A \otimes \mathcal B)^+$. If $A, B$ are hermitian elements of $\mathcal{A}$ and $A\leq B$, then for every positive element $C$ of $\mathcal{B}$, we have $A \otimes C \leq B \otimes C$.\\
If $T_1$ and $T_2$ are two maps on $\mathcal H$ and $\mathcal K$, respectively, then the tensor product $T_1$ and $T_2$ on $\mathcal H\otimes \mathcal K$ is defined by $(T_1\otimes T_2)(f\otimes g)=T_1f\otimes T_2g$ for $f\otimes g\in \mathcal H \otimes \mathcal K $. For $T_1\in Hom_{\mathcal A}^*(\mathcal H)$ and $T_2\in Hom_{\mathcal A}^*(\mathcal K)$,  it is a routine verification that $T_1^*\otimes T_2^*$ is the adjoint of $T_1\otimes T_2$, so in fact $T_1\otimes T_2\in Hom^*_{\mathcal A\otimes \mathcal B}(\mathcal H \otimes \mathcal K )$. For more details one can see chapter 4 of \cite{L}.\\
Many useful techniques in Hilbert spaces are either not available or not known in Hilbert C$^*$-modules. For example, it is well-known that every Hilbert space has an orthonormal basis but this is not true for every Hilbert C$^*$-modules and the analogue of the Riesz representation theorem for bounded $\mathcal A$-linear mapping is not valid for Hilbert $\mathcal A$-module $\mathcal H $ but this is true for self-dual Hilbert modules \cite{MS}. Note that a Hilbert $\mathcal A$-module $\mathcal H$ is called self-dual if $\mathcal H\cong\mathcal H^{'}$, where $\mathcal H^{'}$ is the set of all bounded $\mathcal A$-linear maps from $\mathcal H$ to $\mathcal A$.\\
The notion of frames for Hilbert spaces had been extended by Frank and Larson to the Hilbert C$^*$-modules and some properties of this frames were also investigated in \cite{FL1, FL2, WJ, KHKH}.\\
Let  $\mathcal A$ be a unital C$^*$-algebra and $J$ be a finite or countable index set. A  sequence $\{f_j\}_{j\in J}$ of elements in a Hilbert $\mathcal A$-module $\mathcal H $ is said to be a (standard) $*$-frame for $\mathcal H$
if there exists strictly nonzero elements $A$ and $B$ of $\mathcal A$ such that
\[
A\langle f, f\rangle A^* \leq \sum_{j\in J}\langle f, f_j\rangle\langle f_j, f\rangle \leq B\langle f, f\rangle B^*,
\]
holds for every $f\in \mathcal H$, where the sum in the middle of the inequality is convergent in norm (see \cite{AD}). The element $A$ is called the lower $*$-frame bound and $B$ called upper $*$-frame bound.  If the right side of this inequality holds for $\{f_j\}_{j\in J}$ then we say that $\{f_j\}_{j\in J}$ is a $*$-Bessel sequence. Trivially every frame for a Hilbert modules is $*$-frame.

If $\mathcal{A}=\mathbb{C}$ then the $*$-frame $\{f_j\}_{j\in J}$ is indeed a frame for the Hilbert module $\mathcal{H}$.  Note that a nonzero element $A$ in unital C$^*$-algebra $\mathcal A$ is called strictly nonzero if zero doesn't belong to $\sigma(A)$, where $\sigma(A)$ is the spectrum of the element $A$.

The following result of Jing in \cite{WJ} is useful for our study.
 \begin{lemma}\label{norm}
$\{f_j\}_{j\in J}$ is a frame of a finitely or countably generated Hilbert $\mathcal A$-module $\mathcal H $ over a unital C$^*$-algebra $\mathcal A$  with frame bounds $A, B$, respectively, if and only if
\[
A\|f\|^2 \leq \|\sum_{j\in J}\langle f, f_j\rangle\langle f_j, f\rangle\| \leq B\|f\|^2, \quad f\in \mathcal{H}.
\]
\end{lemma}

 Suppose that $\{f_j\}_{j\in J}$ is $*$-frame of a finitely or countably generated Hilbert $\mathcal A$-module $\mathcal H $ over a unital C$^*$-algebra $\mathcal A$. The operator $U:\ell^2(\mathcal A)\rightarrow \mathcal H $ defined by $U(\{g_j\}_{j\in J})=\sum_{j\in J}g_jf_j$ is called the  synthesis operator. The adjoint operator
 $U^*:\mathcal H \rightarrow \ell^2(\mathcal A)$ is given by $U^*(f)=\{\langle f, f_j\rangle\}_{j\in J}$ is called analysis operator. By composing $U$ and $U^*$ we obtain the $*$-frame operator $S:\mathcal H \rightarrow \mathcal H $ by $S(f)=UU^*(f)=\sum_{j\in J}\langle f, f_j\rangle f_j$ (see \cite{AD, OC}).\\
  In this paper we are going to study $*$-atomic system in Hilbert C$^*$-module. In section 2 we introduce a family of local $*$-atoms for $\mathcal H $ and an $*$-atomic system for an adjointable operator over $\mathcal H $, then their relations are studied. Next in section 3 we give some relations between $*$-atomic systems and a generalization of $K$-frames for Hilbert C$^*$-modules that namely $*$-$K$-frame. In section 4 some properties of $*$-$K$-frame are studied in particular $*$-$K$-frames for operators on tensor product of Hilbert modules are discussed. In the last section some perturbation results for $*$-$K$-frames are established.\\
  Throughout this paper $\mathcal H$, $\mathcal A$, $1_{\mathcal A}$, $\mathcal Z(\mathcal A)$, $U$ and $U^*$ denoted a finitely or countably generated Hilbert module, unital C$^*$-algebra, unit of the C$^*$-algebra $\mathcal A$, centre of the C*-algebra $\mathcal A$, synthesis and analysis operators of given $*$-frames, respectively. Given an operator $T$, we denote its range by $R(T)$.
\section{Local $*$-Atoms and $*$-Atomic Systems}
In this section first we introduce a family of local $*$-atoms for $\mathcal H $ and then a $*$-atomic systems for an adjointable operator over $\mathcal H$ as a generalization of local $*$-atoms will be considered and some of their properties are studied.
\begin{definition}
Suppose that $\mathcal H $ is a finitely or countably generated Hilbert module over a unital C$^*$-algebra $\mathcal A$ and $\{f_j\}_{j\in J}\subseteq \mathcal H $ is a $*$-Bessel sequence and $\mathcal H _0$ is a closed submodule of $\mathcal H $. $\{f_j\}_{j\in J}$ is called a family of local $*$-atoms for $\mathcal H _0$ if there exists a sequence $\{c_j\}_{j\in J}$ of adjointable operators $c_j : \mathcal H _0\rightarrow \mathcal A$ such that for every $f\in \mathcal H _0$ \\
$(i)$ there exists strictly nonzero element $C\in \mathcal A$ with $\sum_{j\in J} (c_j(f))(c_j(f))^* \leq C\langle f, f\rangle C^*$, \\
$(ii)$ $f=\sum_{j\in J}c_j(f)f_j$.
\end{definition}
Trivially every $*$-frame for the Hilbert module $\mathcal H$ is a local $*$-atom for $\mathcal H_0=\mathcal H$. But its converse isn't true in general. Note that if $\{f_j\}_{j\in J}$ is a $*$-frame for a Hilbert $\mathcal A$-module $\mathcal H$ so is local $*$-atom for $\mathcal H_0=\mathcal H$ with $c_j(f)=\langle f, S^{-1}f_j\rangle$, where $S$ is the $*$-frame operator.
\begin{proposition}
Let $\mathcal H _0$ be an orthogonally complemented submodule of $\mathcal H$. Suppose that \(\{f_j\}_{j\in J}\subseteq \mathcal H \) is a family of local $*$-atoms for $\mathcal H _0$ then $\{P_{\mathcal H _0}f_j\}_{j\in J}$ is a $*$-frame for $\mathcal H _0$, when $P_{\mathcal H _0}$ is the orthogonal projection of $\mathcal H $ onto $\mathcal H _0$.
\end{proposition}
\begin{proof}
It is enough to show that $\{P_{\mathcal H _0}f_j\}_{j\in J}$ has a lower $*$-frame bound. By definition of local $*$-atoms there exists a sequence $\{c_j\}_{j\in J}$ of adjointable operators $c_j : \mathcal H _0\rightarrow \mathcal A$ such that $f=\sum_{j\in J}c_j(f)f_j$ and also there exists strictly nonzero element $C\in \mathcal A$ with $\sum_{j\in J} (c_j(f))(c_j(f))^* \leq C\langle f, f\rangle C^*$, for every $f\in \mathcal H _0$. Thus for every $f\in \mathcal H _0$,
\begin{eqnarray*}
||f||^4&=&||\langle \sum_{j\in J}c_j(f)f_j, f\rangle||^2\\
&=&||\sum_{j\in J}c_j(f)\langle f_j, f\rangle||^2\\
&\leq&||\sum_{j\in J}(c_j(f))(c_j(f))^*||~||\sum_{j\in J}\langle f, f_j\rangle \langle f_j, f\rangle||\\
&\leq&\|C\|^2||f||^2||\sum_{j\in J}\langle f, f_j\rangle \langle f_j, f\rangle||\\
&=& \|C\|^2 \|f\|^2 \|\sum_{j\in J}\langle P_{\mathcal H_0}f, f_j\rangle \langle f_j, P_{\mathcal H_0}f\rangle\|\\
&=& \|C\|^2 \|f\|^2 \|\sum_{j\in J}\langle f, P_{\mathcal H_0}f_j\rangle \langle P_{\mathcal H_0}f_j, f\rangle\|.
\end{eqnarray*}
Hence by Lemma \ref{norm}
\[\frac{1_{\mathcal A}}{\|C\|}\langle f, f\rangle (\frac{1_{\mathcal A}}{\|C\|})^*\leq \sum_{j\in J}\langle f, P_{\mathcal H_0}f_j\rangle \langle P_{\mathcal H_0}f_j, f\rangle.\]
 So  $\{P_{\mathcal H _0}f_j\}_{j\in J}$ is a $*$-frame for $\mathcal H _0$ with the lower frame bound $\frac{1_{\mathcal A}}{\|C\|}$.
\end{proof}
Here $*$-atomic system, a generalization of atomic systems for operators, is introduced.
\begin{definition}
Let $\mathcal H $ be a finitely or countably generated Hilbert module over a unital C$^*$-algebra $\mathcal A$ with unit $1_{\mathcal A}$ and $K\in Hom^*_{\mathcal A}(\mathcal H)$. We say that $\{f_j\}_{\j\in J}\subseteq \mathcal H $ is a $*$-atomic system for $K$ if $\{f_j\}_{\j\in J}$ is a $*$-Bessel sequence and there exists a strictly nonzero $C\in \mathcal A$ such that for all $f\in \mathcal H $ there is
$a_f=\{a_j\}_{j\in J}\in \ell^2(\mathcal A)$ such that $K(f)=\sum_{j\in J}a_jf_j$ and $\langle a_f, a_f\rangle\leq C\langle f, f\rangle C^*$.
\end{definition}
\begin{example} (See \cite{AD})
Let $\ell^{\infty}$ be the unitary C$^*$-algebra of all bounded complex-valued sequences. Let $c_0$ be the set of all sequences converging to zero. Then $c_0$ is a Hilbert $\ell^{\infty}$-module with $\ell^{\infty}$-valued inner product $\langle u, v\rangle=\{u_i\overline{v_i}\}_{i\in \mathbb N}$, for $u, v\in c_0$. For any $j\in \mathbb N$ let $f_j=\{f_i^j\}_{i\in \mathbb N}\in c_0$ be defined by
\begin{equation*}
 f_i^j=\left\{\begin{array}{l c} \frac{1}{3}+\frac{1}{i}& i=j\\ 0& i\neq j\end{array}\right. , \quad j\in \mathbb N
\end{equation*}
and define $K:c_0\to c_0$ by $K(u)=\{\sum_{j\in \mathbb N} u_i|f_i^j|^2\}_{i\in \mathbb N}$ for $u=\{u_i\}_{i\in \mathbb N}\in c_0$, then $K\in Hom_{\ell^{\infty}}^*(c_0)$. So $\{f_j\}_{j\in \mathbb N}$ is a $*$-atomic system for $c_0$. Note that in this case $a_f=\{u\overline{f_j}\}_{j\in \mathbb N}\in \ell^2(\ell^{\infty})$ and $C=\{\frac{1}{3}+\frac{1}{i}\}_{i\in \mathbb N}\in \ell^{\infty}$, since
\[ K(u)=\{\sum_{j\in \mathbb N} u_i|f_i^j|^2\}_{i\in \mathbb N}=\sum_{j\in \mathbb N}\{u_i|f_i^j|^2\}_{i\in \mathbb N}=\sum_{j\in J}\{u_i\overline{f_i^j}\}_{i\in \mathbb N}f_j=\sum_{j\in J}u\overline{f_j}f_j\]
and
 \[\langle a_f, a_f\rangle=\sum_{j\in J}u\overline{f_j}f_j\overline{u}=\{|u_i|^2(\frac{1}{3}+\frac{1}{i})^2\}_{i\in \mathbb N}=\{\frac{1}{3}+\frac{1}{i}\}_{i\in \mathbb N}\langle u, u\rangle \{\frac{1}{3}+\frac{1}{i}\}_{i\in \mathbb N}.\]
\end{example}
In the following proposition, which can be proved easily, some relations of families of local $*$-atoms and  $*$-atomic systems are stated.
\begin{proposition}
Let $\mathcal H _0$ be an orthogonally complemented submodule of $\mathcal H$ and $\{f_j\}_{j\in J}\subseteq \mathcal H $ be a $*$-Bessel sequence then the following are equivalent.\\
$(i)$ $\{f_j\}_{j\in J}$ is a family of local $*$-atoms for $\mathcal H _0$.\\
$(ii)$ $\{f_j\}_{j\in J}$ is a $*$-atomic system for $P_{\mathcal H _0}$, where $P_{\mathcal H _0}$ is an orthogonal projection from $\mathcal H $ onto $\mathcal H _0$.
\end{proposition}
\section{$*$-Atomic Systems and $*$-$K$-Frames}
In this section we study a $*$-atomic systems and its relation with $*$-$K$-frames on Hilbert modules discussed.
\begin{definition}
Let $K\in Hom^*_{\mathcal A}(\mathcal H)$. A sequence $\{f_j\}_{j\in J}\subseteq \mathcal H$ is called a $*$-$K$-frame If there exists strictly nonzero $A, B\in \mathcal A$ such that
\begin{eqnarray}\label{AB}
A\langle K^*f, K^*f\rangle A^* \leq \sum_{j\in J}\langle f, f_j\rangle\langle f_j, f\rangle\leq B\langle f, f\rangle B^*,
\end{eqnarray}
holds for every $f\in \mathcal H$.
\end{definition}
Every $*$-frames is a $*$-$K$-frames. Indeed for any $K\in Hom^*_{\mathcal A}(\mathcal H)$ we have
\begin{eqnarray}\label{kk}
\langle K^*f, K^*f\rangle\leq\|K\|^2 \langle f, f\rangle, \quad f\in \mathcal H
\end{eqnarray}
Now if $\{f_j\}_{j\in J}$ is a $*$-frame with bounds $A$ and $B$ then by \eqref{kk} and the fact that for $A, B\in \mathcal A$ the inequality $A\leq B$ implies that $CAC^*\leq CBC^*$, for $C\in \mathcal A$, we have
 \[(A\|K\|^{-1})\langle K^*f, K^*f\rangle (A\|K\|^{-1})^*\leq A\langle f, f\rangle A^*\leq\sum_{j\in J}\langle f,f_j\rangle\langle f_j, f\rangle\leq B\langle f, f\rangle B^*.\]
 Therefore $\{f_j\}_{j\in J}$ is a $*$-$K$-frame with $*$-frame bounds $A\|K\|^{-1}$ and $B$.\\
\begin{lemma}\label{norm1}
Suppose that $\mathcal H$ is a Hilbert modules over a C$^*$-algebra $\mathcal A$. If $\{f_j\}_{j\in J}$ is a $*$-$K$-frame with $*$-frame bounds $A$ and $B$ then
\begin{eqnarray}\label{normeq}
\|AK^*f\|^2\leq \|\sum_{j\in J}\langle f, f_j\rangle\langle f_j, f\rangle\|\leq \|Bf\|^2, \quad f\in \mathcal H.
\end{eqnarray}
For the converse suppose that \eqref{normeq} holds, for any $A, B\in \mathcal Z(\mathcal A)$. Let $T:\mathcal H\to \ell^2(\mathcal A)$ defined by $Tf=\{\langle f, f_j\rangle\}_{j\in J}$ and $\overline{R(T)}$ be orthogonally complemented. Then $\{f_j\}_{j\in J}$ is a $*$-$K$-frame with $*$-frame bounds $mA$ and $nB$, for some positive real numbers $m, n$.
\end{lemma}
\begin{proof}
$\Rightarrow$) The first part is obvious. For the reverse part, first notice that for any $f\in \mathcal{H}$,
\[\|Tf\|^2=\|\langle Tf, Tf\rangle\|=\|\sum_{j\in J}\langle f, f_j\rangle\rangle f_j, f\rangle\|\leq \|Bf\|^2\leq \|B\|^2\|f\|^2,\]
so $\|Tf\|\leq \|B\|\|f\|$. Therefore $T$ is bounded. Also it is not hard to see that $T$ is linear and adjointable and its adjoint is $T^*(\{g_j\}_{j\in J})=\sum_{j\in J}g_jf_j$, $\{g_j\}_{j\in J}\in \ell^2(\mathcal A)$.

Now for each $B\in \mathcal Z(\mathcal A)$, the mapping $Q_B:\mathcal H\to \mathcal H$ defined by $Q_Bf=Bf$ has the adjoint $Q_{B^*}$, since
\[\langle Q_Bf, g\rangle=\langle Bf, g\rangle=B\langle f, g\rangle=\langle f, g\rangle B=\langle f, B^*g\rangle=\langle f, Q_{B^*}g\rangle, \quad f, g\in \mathcal H.\]
Therefore \eqref{normeq} is equivalent to
\[\|AK^*f\|^2\leq \|Tf\|^2\leq \|Q_Bf\|^2\ \ \ \ f\in\mathcal{H}. \]
By Theorem \ref{adj} there exists $\lambda, \mu>0$ such that for every $f\in \mathcal{H}$,
\[\sqrt{\lambda}A\langle K^*f, K^*f\rangle (\sqrt{\lambda}A)^*\leq\sum_{j\in J}\langle f, f_j\rangle\langle f_j, f\rangle\leq \sqrt{\mu}B\langle f, f\rangle (\sqrt{\mu}B)^*.\]
Letting $m=\sqrt{\lambda}$ and $n=\sqrt{\mu}$, we complete the proof.
\end{proof}
By a similar argument to the proof of Lemma \ref{norm1}one may prove the following lemma by Theorem \ref{adj2}.
\begin{lemma}\label{norm2}
Suppose that $\mathcal H$ is a Hilbert modules over a C$^*$-algebra $\mathcal A$. If $\{f_j\}_{j\in J}$ is a $*$-$K$-frame with $*$-frame bounds $A$ and $B$ then
\begin{eqnarray}\label{normeq1}
\|AK^*f\|^2\leq \|\sum_{j\in J}\langle f, f_j\rangle\langle f_j, f\rangle\|\leq \|Bf\|^2, \quad f\in \mathcal H.
\end{eqnarray}
For the converse suppose that \eqref{normeq1} holds for any $A, B\in \mathcal Z(\mathcal A)$ and the operator $T:\mathcal H\to \ell^2(\mathcal A)$ defined by $Tf=\{\langle f, f_j\rangle\}_{j\in J}$ has closed range. Then $\{f_j\}_{j\in J}$ is a $*$-$K$-frame with $*$-frame bounds $mA$ and $nB$, for some positive real numbers $m, n$.
\end{lemma}
In the following theorem it is proved that $*$-atomic systems for an operator $K$ are indeed $*$-$K$-frame.
\begin{theorem}\label{main}
Let $K\in Hom^*_{\mathcal A}(\mathcal H)$ and $\{f_j\}_{j\in J}\subseteq \mathcal H$ be a $*$-Bessel sequence.
Suppose that $T:\ell^2(\mathcal A)\rightarrow \mathcal H$ is defined by $T(\{g_j\}_{j\in J})=\sum_{j\in J}g_jf_j$ and $\overline{R(T^*)}$ is orthogonally complemented. Then $\{f_j\}_{j\in J}$ is a $*$-atomic system for $K$ if and only if $\{f_j\}$ is a $*$-atomic system. Moreover, in this case if $\mathcal A$ is finite dimensional then there exists another $*$-Bessel sequence $\{h_j\}_{j\in J}$ such that for all $f\in \mathcal H$,
\[K(f)=\sum_{j\in J}\langle f, h_j\rangle f_j.\]
\end{theorem}
\begin{proof}
Suppose that $\{f_j\}_{j\in J}$ is a $*$-atomic system for $K$. For any $f\in \mathcal H $ we have
\[\|K^*f\|^2=\sup_{\|g\|=1}\|\langle g, K^*f\rangle\|^2=\sup_{\|g\|=1}\|\langle Kg, f\rangle\|^2.\]
By definition of $*$-atomic system, there is $b_g=\{B_j\}_{j\in J}\in \ell^2(\mathcal A)$ such that $K(g)=\sum_{j\in J}B_jf_j$ and $\langle b_g, b_g\rangle\leq C\langle g, g\rangle C^*$, for some strictly nonzero $C\in \mathcal A$. Thus
\begin{eqnarray*}
\|K^*f\|^2&=&\sup_{\|g\|=1}\|\langle\sum_{j\in J}B_jf_j, f\rangle\|^2\\
&=&\sup_{\|g\|=1}\|\sum_{j\in J}B_j\langle f_j, f\rangle\|^2\\
&\leq&\sup_{\|g\|=1}\|\sum_{j\in J}B_jB_j^*\|\|\sum_{j\in J}\langle f, f_j\rangle\langle f_j, f\rangle\|\\
&\leq&\|C\|^2\|\sum_{j\in J}\langle f, f_j\rangle\langle f_j, f\rangle\|.
\end{eqnarray*}
Note that the last inequality holds by the fact that in a C$^*$-algebra $\mathcal A$, if $A, B\in \mathcal A$ and $0\leq A\leq B$ then $\|A\|\leq \|B\|$. Hence
\[\frac{1}{\|C\|^2}\|K^*f\|^2\leq \|\sum_{j\in J}\langle f, f_j\rangle\langle f_j, f\rangle\|,\ \ \ f\in\mathcal{H}.\]
So by Theorem \ref{adj} there exists positive real number $\mu>0$ such that $\mu KK^*\leq TT^*$. Therefore
\[(\sqrt{\mu}1_{\mathcal A})\langle K^*f, K^*f\rangle(\sqrt{\mu}1_{\mathcal A})^*\leq \sum_{j\in J}\langle f, f_j\rangle\langle f_j, f\rangle. \]
For the converse, suppose that there exists strictly nonzero $A, B\in \mathcal A$ such that \eqref{AB} holds. So,
$\|A^{-1}\|^{-2}\|K^*f\|^2\leq\|T^*f\|^2$. Thus By Theorem \ref{adj} there exists an adjointable operator $Q:\mathcal H \rightarrow \ell^2(\mathcal A)$ such that $K=TQ$. Since $Qf\in \ell^2(\mathcal A)$ we set $a_f=Qf=\{A_j\}_{j\in J}$. Thus
\[K(f)=T(Q(f))=T(\{a_j\}_{j\in J})=\sum_{j\in J}A_j f_j.\]
It is enough to show that $\langle a_f, a_f\rangle\leq C\langle f, f\rangle C^*$ for some strictly nonzero $C\in \mathcal A$. Since $Q$ is adjointable then we have
\[\langle a_f, a_f\rangle=\langle Qf, Qf\rangle\leq \|Q\|^2\langle f, f\rangle=(\|Q\|1_{\mathcal A})\langle f, f\rangle(\|Q\|1_{\mathcal A})^*\]
which implies that $\langle a_f, a_f\rangle\leq C\langle f, f\rangle C^*$ with $C=\|Q\|1_{\mathcal A}$. Now let $\mathcal A$ be finite dimensional then $\|A^{-1}\|^{-2}\|K^*f\|^2\leq\|T^*f\|^2$ thus by Theorem \ref{adj} there exists an adjointable operator $Q:\mathcal H \rightarrow \ell^2(\mathcal A)$ such that $K=TQ$. In this case $\ell^2(\mathcal A)$ is self dual so there exists $\{h_j\}_{j\in J}$ for which $Q(f)=\{\langle f, h_j\rangle\}_{j\in J}$. Thus
\[K(f)=T(Q(f))=T(\{\langle f, h_j\rangle\}_{j\in J})=\sum_{j\in J}\langle f, h_j\rangle f_j.\]
and
\[\sum_{j\in J}\langle f, h_j\rangle\langle h_j, f\rangle=\langle Qf, Qf\rangle\leq\|Q\|^2\langle f, f\rangle=(\|Q\|1_{\mathcal A})\langle f, f\rangle(\|Q\|1_{\mathcal A})^*.\]
Hence $\{h_j\}_{j\in J}$ is a $*$-Bessel sequence.
\end{proof}
\begin{corollary}
Let $K\in Hom^*_{\mathcal A}(\mathcal H)$ and $\{f_j\}_{j\in J}\subseteq \mathcal H$ be a $*$-Bessel sequence with bound $B$.
Suppose that $T:\ell^2(\mathcal A)\rightarrow \mathcal H$ is defined by $T(\{g_j\}_{j\in J})=\sum_{j\in J}g_jf_j$ and $\overline{R(T^*)}$ is orthogonally complemented. Then the following are equivalent.\\
 $(i)$ $\{f_j\}_{j\in J}$ is a $*$-atomic system for $K$ such that the strictly nonzero $C$ in the definition of $*$-atomic system is in $\mathcal Z(\mathcal A)$,\\
 $(ii)$ There exists positive real numbers $m>0$ such that for any $f\in \mathcal H$
 \begin{eqnarray*}
(\frac{m}{C})\langle K^*f, K^*f\rangle (\frac{m}{C})^* \leq \sum_{j\in J}\langle f, f_j\rangle\langle f_j, f\rangle\leq B\langle f, f\rangle B^*.
\end{eqnarray*}
\begin{proof}
The proof obtains by Lemma \ref{norm1} and Theorem \ref{main}.
\end{proof}
\end{corollary}
 We give another characterization of $*$-atomic systems in the next theorem.
 \begin{theorem}
Let $K\in Hom^*_{\mathcal A}(\mathcal H)$ and $\{f_j\}_{j\in J}\subseteq \mathcal H$ be a $*$-Bessel sequence.
Suppose that $T:\ell^2(\mathcal A)\rightarrow \mathcal H$ is defined by $T(\{g_j\}_{j\in J})=\sum_{j\in J}g_jf_j$ and $\overline{R(T^*)}$ is orthogonally complemented. Then $\{f_j\}_{j\in J}$ is a $*$-atomic systems if and only if there exists an adjointable  operator $L:\ell^2(\mathcal A)\rightarrow \mathcal H$ for which $L(e_j)=f_j$, for every $j\in J$, where $\{e_j\}_{j\in J}$ is the standard orthonormal basis for $\ell^2(\mathcal A)$, and $R(K)\subseteq R(L)$.
 \end{theorem}
\begin{proof}
Suppose that $\{f_j\}_{j\in J}$ is a $*$-atomic system for $K$. Consider the mapping $T_1:\mathcal H \rightarrow \ell^2(\mathcal A)$ defined by $T_1(f)=\{\langle f, f_j\rangle\}_{j\in J}$.
The operator $T_1$ is well-defined, adjointable and $T_1^*e_j=f_j$, for $j\in J$, since for every $h\in \mathcal H $ we have
\[ \langle f_j, h\rangle=\langle e_j, T_1h\rangle=\langle T_1^*e_j, h\rangle.\]
Let $L=T_1^*$ then by Theorem \ref{main1} $R(K)\subseteq R(L)$.\\
For the converse if $R(K)\subseteq R(L)$ then by Theorem \ref{adj} there exists an adjointable operator $Q$ such that $K=LQ$. Thus for any $f\in \mathcal H$
\[K(f)=LQ(f)=L(\{C_j\}_{j\in J}),\]
for some $\{C_j\}_{j\in J}\in \ell^2(\mathcal A)$. Put $c_f=Q(f)=\{C_j\}_{j\in J}$. For every $g=\{g_j\}_{j\in J}\in \ell^2(\mathcal A)$ we have
\[\langle L^*f, g\rangle=\langle L^*f, \sum_{j\in J}\langle g_j, e_j\rangle e_j\rangle
=\sum_{j\in J}\langle e_j, g\rangle \langle f, Le_j\rangle=\langle \sum_{j\in J}\langle f, f_j\rangle e_j, g\rangle.\]
So $L^*(f)=\sum_{j\in J}\langle f, f_j\rangle e_j$. Hence
\[ \langle L^*f, g\rangle=\langle \sum_{j\in J}\langle f, f_j\rangle e_j, g\rangle=\langle \{\langle f, f_j\rangle\}, \{g_j\}\rangle
=\sum_{j\in J}\langle f, f_j\rangle g_j^*=\langle f, \sum_{j\in J}g_jf_j\rangle.\]
So $L(g)=\sum_{j\in J}g_jf_j$. Therefore $L(\{C_j\}_{j\in J})=\sum_{j\in J}C_jf_j$ and $\langle c_f, c_f\rangle=\langle Qf, Qf\rangle\leq \|Q\|^2\langle f, f\rangle=(\|Q\|1_{\mathcal A})\langle f, f\rangle(\|Q\|1_{\mathcal A})^*$.
\end{proof}
\begin{corollary}\label{main1}
Let $K\in Hom^*_{\mathcal A}(\mathcal H)$. Suppose that $\{f_j\}_{j\in J}$ is a $*$-Bessel sequence and $\overline{R(U^*)}$ is an orthogonally complemented. Then $\{f_j\}_{j\in J}$ is a $*$-atomic systems for $K$ if and only if $\{f_j\}_{j\in J}$ is a $*$-Bessel sequence and $R(K)\subseteq R(U)$.
\end{corollary}
 \section{Some more Properties of $*$-$K$-frames}
 In this section first using a $*$-$K$-frames and some elements of $Hom^*_{\mathcal A}(\mathcal H )$, we are going to construct new $*$-$K$-frames. Next the tensor product of two $*$-$K$-frames and $*$-L-frames are considered.
\begin{proposition}\label{propos}
Let $K, L\in Hom^*_{\mathcal A}(\mathcal H)$ and $\{f_j\}_{j\in J}$ be a $*$-$K$-frame with the $*$-frame bounds $A, B$, then\\
$(i)$ If  $T : \mathcal H \rightarrow \mathcal H $ is an co-isometry such that $KT=TK$ then $\{Tf_j\}_{j\in J}$ is a  $*$-$K$-frame with the same $*$-frame bounds.\\
$(ii)$ $\{Lf_j\}_{j\in J}$ is a $*$-LK-frame with the $*$-frame bounds $A$ and $B||L||$, respectively.\\
$(iii)$ For any $n\in \mathbb{N}$, $\{L^nf_j\}_{j\in J}$ is a $*$-L$^n$K-frame.\\
$(iv)$ If $R(L)\subseteq R(K)$ such that $K$ has closed range then $\{f_j\}_{j\in J}$ is also a $*$-$L$-frame.
\end{proposition}
\begin{proof}
Since $\{f_j\}_{j\in J}$ is a $*$-$K$-frame with the $*$-frame bounds $A, B$, so we have
\[
A\langle K^*f, K^*f\rangle A^*\leq\sum_{j\in J}\langle f, f_j\rangle\langle f_j, f\rangle\leq B\langle f, f\rangle B^*, ~~f\in \mathcal H .\]
Hence for any $f\in \mathcal H $,
\begin{eqnarray*}
 \sum_{j\in J}\langle f, Tf_j\rangle\langle Tf_j, f\rangle\leq B\langle T^*f, T^*f\rangle B^*=B\langle f, f\rangle B^*.
\end{eqnarray*}
On the other hand for all $f\in \mathcal H $
\begin{eqnarray*}
\sum_{j\in J}\langle f, Tf_j\rangle\langle Tf_j, f\rangle &\geq& A\langle K^*T^*f, K^*T^*f\rangle A^*\\
&=& A\langle T^*K^*f, T^*K^*f\rangle A^*\\
&=& A\langle K^*f, K^*f\rangle A^*,
\end{eqnarray*}
which proves $(i)$.\\
For proving $(ii)$, one may see that for any $f\in \mathcal H $,
\begin{eqnarray*}
A\langle (LK)^*f, (LK)^*f\rangle A^*=A\langle K^*L^*f, K^*L^*f\rangle A^*&\leq&\sum_{j\in J}\langle f, Lf_j\rangle\langle Lf_j, f\rangle\\
&\leq& B\langle L^*f, L^*f\rangle B^*\\
&\leq& (B\|L\|)\langle f, f\rangle (B\|L\|)^*.
\end{eqnarray*}
$(iii)$ is trivial by applying $(ii)$.\\
For proving $(iv)$, if $A$ and $B$ are the $*$-$K$-frame bounds of $\{f_j\}_{j\in J}$ then  by the fact that $R(L)\subseteq R(K)$ with closed range $K$ and Theorem \ref{adj2}, there exists positive real number $\lambda>0$ such that for all $f\in \mathcal H$, $\lambda LL^*f\leq KK^*f$. Thus for any $f\in H$,
\[(\sqrt{\lambda}A)\langle L^*f, L^*f\rangle(\sqrt{\lambda}A)^*\leq A\langle K^*f, K^*f\rangle A^*\leq\sum_{j\in J}\langle f, f_j\rangle \langle f_j, f\rangle\leq B\langle f, f\rangle B^* .\]
\end{proof}
\begin{proposition}
Let $K\in Hom^*_{\mathcal A}(\mathcal H)$ and  $\{f_j\}_{j\in J}$ be a $*$-frame with the $*$-frame bounds $A, B$,
then $\{Kf_j\}_{j\in J}$ is a $*$-$K$-frame  with the $*$-frame bounds $A, B||K||$. The $*$-frame operator of $\{Kf_j\}_{j\in J}$ is $S^{'}=KSK^*$, where $S$ is the $*$-frame operator of $\{f_j\}_{j\in J}$.
\end{proposition}
\begin{proof}
Since $\{f_j\}_{j\in J}$ is a $*$-frame so for any $f\in \mathcal H $,
\[
A\langle K^*f, K^*f\rangle A^*\leq\sum_{j\in J}\langle f, Kf_j\rangle\langle Kf_j, f\rangle\leq B\langle K^*f, K^*f\rangle B^*\leq (B\|K\|)\langle f, f\rangle(B\|K\|)^*.
\]
But by definition of $S$, $SK^*f=\sum_{j\in J}\langle f, Kf_j\rangle f_j$. Thus
\begin{eqnarray}\label{frame op}
KSK^*f=K\sum_{j\in J}\langle f, Kf_j\rangle f_j=\sum_{j\in J}\langle f, Kf_j\rangle Kf_j.
\end{eqnarray}
Hence $S^{'}=KSK^*$.
\end{proof}
\begin{proposition}
Suppose that $K\in Hom^*_{\mathcal A}(\mathcal H)$ and  $\{f_j\}_{j\in J}$ is a $*$-frame, then $\{KS^{-1}f_j\}$ is a $*$-$K$-frame, when $S$ is the $*$-frame operator of $\{f_j\}_{j\in J}$.
\end{proposition}
\begin{proof}
By Theorem \ref{main}, it is enough to show that $\{f_j\}_{j\in J} $ is a $*$-atomic system for $\mathcal H $. We have that $f= \sum_{j\in J}\langle f, f_j\rangle S^{-1}f_j$, for all $f\in \mathcal H $. Thus
\begin{eqnarray*}
Kf=\sum_{j\in J}\langle f, f_j\rangle KS^{-1}f_j, ~~ f\in \mathcal H .
\end{eqnarray*}
Since $\{f_j\}_{j\in J}$ is a $*$-frame so $\{\langle f, f_j\rangle\}_{j\in J}$ is a $*$-Bessel sequence and trivially $\{KS^{-1}f_j\}_{j\in J}$ is a $*$-Bessel sequence, since for $f\in \mathcal H $,
\begin{eqnarray*}
\sum_{j\in J}\langle f, KS^{-1}f_j\rangle \langle KS^{-1}f_j, f\rangle&\leq& B\langle (KS^{-1})^*f, (KS^{-1})^*f\rangle B^*\\
 &\leq& (B\|S^{-1}\|\|K\|)\langle f, f\rangle (B\|S^{-1}\|\|K\|)^*.
\end{eqnarray*}
\end{proof}
\begin{proposition}
If $K, L\in Hom^*_{\mathcal A}(\mathcal H)$ and $\overline{R(L^*)}$ is orthogonally complemented, $R(K)\subseteq R(L)$ and  $\{f_j\}_{j\in J}$ is a $*$-frame with $*$-frame bounds $A, B$, then $\{Lf_j\}_{j\in J}$ is a $*$-$K$-frame with the $*$-frame operator  $S^{'}=L^*SL$.
\end{proposition}
\begin{proof}
Since $R(K)\subseteq R(L)$ so by Theorem \ref{adj} there exists a positive real number $\lambda>0$ such that $\lambda KK^*\leq LL^*$ therefore $\lambda\langle K^*f, K^*f\rangle\leq \langle L^*f, L^*f\rangle$ hence $(A\sqrt{\lambda})\langle K^*f, K^*f\rangle (A\sqrt{\lambda})^*\leq A\langle L^*f, L^*f\rangle A^*$, thus by the facts that $\{f_j\}_{j\in J}$ is a $*$-frame we have
\begin{eqnarray*}
(A\sqrt{\lambda})\langle K^*f, K^*f\rangle (A\sqrt{\lambda})^*\leq A\langle L^*f, L^*f\rangle A^*&\leq&\sum_{j\in J}\langle f, Lf_j\rangle\langle Lf_j, f\rangle\\
&\leq& B\langle L^*f, L^*f\rangle B^*\\
&\leq& (B\|L\|)\langle f, f\rangle (B\|L\|)^*.
\end{eqnarray*}
So $\{Lf_j\}_{j\in J}$ is a $*$-$K$-frame with $*$-frame bounds $A\sqrt{\lambda}$ and $B\|L\|$. The proof of  $S^{'}=L^*SL$ is obvious.
\end{proof}
In the rest of this section we are going to study the tensor product of two $*$-$K$-frame and $*$-$L$-frame.
\begin{theorem}\label{tensor}
 Let $\{f_j\}_{j\in J}\subseteq \mathcal  H $ and $\{h_j\}_{j\in J}\subseteq \mathcal K $ be two $*$-$K$-frame and $*$-$L$-frame for $\mathcal H$ and $\mathcal K$ with $*$-frame operators $S_f$ and $S_h$ and $*$-frame bounds $A, B$ and $C, D$, respectively. Then $\{f_j\otimes h_j\}_{j\in J}$ is a $*$-$K\otimes L$-frame for Hilbert $\mathcal A\otimes\mathcal B$-module $\mathcal H\otimes\mathcal K $ with $*$-frame operator $S_f\otimes S_h$ and lower $*$-frame bound $A\otimes C$ and upper $*$-frame bound $B\otimes D$.
\end{theorem}
\begin{proof}
Since $\{f_j\}_{j\in J}$ and $\{h_j\}_{j\in J}$ are $*$-$K$-frame and $*$-$L$-frame, respectively, so for any $f\in \mathcal H$ and $h\in \mathcal K$ we have
\begin{eqnarray*}
\sum_{j\in J}\sum_{i\in I}\langle f\otimes h, f_j\otimes h_i\rangle\langle f_j\otimes h_i, f\otimes h \rangle
&=& \sum_{j\in J}\langle f, f_j\rangle\langle f_j, f\rangle\otimes \sum_{i\in I}\langle h, h_i\rangle\langle h_i, h\rangle\\
&\geq& A\langle K^*f, K^*f\rangle A^*\otimes C\langle L^*h, L^*h\rangle C^*\\
&=& (A\otimes C) (\langle K^*f, K^*f\rangle\otimes\langle L^*h, L^*h\rangle) (A\otimes C)^*\\
&=& (A\otimes C) \langle (K^*\otimes L^*)(f\otimes h), (K^*\otimes L^*)(f\otimes h)\rangle(A\otimes C)^*\\
&=& (A\otimes C) \langle (K\otimes L)^*(f\otimes h), (K\otimes L)^*(f\otimes h)\rangle(A\otimes C)^*
\end{eqnarray*}
The rest of the proof is similar to the proof of Theorem 2.2 of \cite{AD}.
\end{proof}
The next corollary is the Theorem 2.2 of \cite{AD}.
\begin{corollary}
Suppose that $\{f_j\}_{j\in J}\subseteq \mathcal H$ and $\{h_j\}_{j\in J}\subseteq \mathcal K $ are  $*$-frames for $\mathcal H$ and $\mathcal K $ with $*$-frame operators $S_f$ and $S_h$ and $*$-frame bounds $A, B$ and $C, D$, respectively. Then $\{f_j\otimes h_j\}_{j\in J}$ is an $*$-frame for Hilbert $\mathcal A\otimes\mathcal B$-module $\mathcal H\otimes\mathcal K$ with $*$-frame operator $S_f\otimes S_h$ and lower $*$-frame bound $A\otimes C$ and upper $*$-frame bound $B\otimes D$.
\end{corollary}
Note that in this case if $\{f_j\}_{j\in J}\subseteq \mathcal H$ and $\{h_j\}_{j\in J}\subseteq \mathcal K $ are frames then this corollary conclude Lemma 3.1 of \cite{KHKH}.
\section{Perturbations of $*$-$K$-frames}
In this section some perturbation of $*$-$K$-frame are studied. The following theorem is a generalization of Theorem 7.1 of \cite{WJ}.
\begin{theorem}\label{pertur1}
Assume that $K, L\in Hom_{\mathcal A}^*(\mathcal H)$ with $R(L)\subseteq R(K)$ and $K$ has closed range. Let $\{f_j\}_{j\in J}$ be a $*$-$K$-frame with $*$-$K$-frame bounds $A, B$. If there exists a constant $M>0$, such that for any $f\in \mathcal H$
\begin{align}\label{min}
\|\sum_{j\in J}\langle f, f_j-h_j\rangle\langle f_j-h_j, f\rangle\|\leq M \min\{ \|\sum_{j\in J}\langle f, f_j\rangle\langle f_j, f\rangle\|, \|\sum_{j\in J}\langle f, h_j\rangle\langle h_j, f\rangle\|\}.
\end{align}
Then $\{h_j\}_{j\in J}$ is a $*$-$L$-frame. The converse is valid for any $K$ which is a co-isometry operator and $R(K)\subseteq R(L)$ with closed range $L$.
\end{theorem}
\begin{proof}
Suppose that $f\in \mathcal H$, so we have
\begin{align}\label{pert1}
\|\sum_{j\in J}\langle f, h_j\rangle\langle h_j, f\rangle\|^{\frac{1}{2}}&=\|\{f, h_j\}\|\leq\|\{f, f_j-h_j\}\|+\|\{f, f_j\}\|\notag\\
&= \|\sum_{j\in J}\langle f, f_j-h_j\rangle\langle f_j-h_j, f\rangle\|^{\frac{1}{2}}+\|\sum_{j\in J}\langle f, f_j\rangle\langle f_j, f\rangle\|^{\frac{1}{2}}\notag\\
&\leq \sqrt{M}\|\sum_{j\in J}\langle f, f_j\rangle\langle f_j, f\rangle\|^{\frac{1}{2}}+\|\sum_{j\in J}\langle f, f_j\rangle\langle f_j, f\rangle\|^{\frac{1}{2}}\notag\\
&=(\sqrt{M}+1)\|\sum_{j\in J}\langle f, f_j\rangle\langle f_j, f\rangle\|^{\frac{1}{2}}\leq \|B\|(1+\sqrt{M})\|f\|.
\end{align}
So by \eqref{pert1} and Lemma \ref{norm} $\{h_j\}_{j\in J}$ is a $*$-Bessel sequence with $*$-Bessel bound $1+\sqrt{M}\|B\|1_{\mathcal A}$.\\
On the other hands we have
\begin{align*}
\|\sum_{j\in J}\langle f, f_j\rangle\langle f_j, f\rangle\|^{\frac{1}{2}}&\leq \|\sum_{j\in J}\langle f, f_j-h_j\rangle\langle f_j-h_j, f\rangle\|^{\frac{1}{2}}+\|\sum_{j\in J}\langle f, h_j\rangle\langle h_j, f\rangle\|^{\frac{1}{2}}\\
&\leq \sqrt{M}\|\sum_{j\in J}\langle f, h_j\rangle\langle h_j, f\rangle\|^{\frac{1}{2}}+\|\sum_{j\in J}\langle f, h_j\rangle\langle h_j, f\rangle\|^{\frac{1}{2}}\notag\\
&=(\sqrt{M}+1)\|\sum_{j\in J}\langle f, h_j\rangle\langle h_j, f\rangle\|^{\frac{1}{2}},
\end{align*}
Since $\{h_j\}_{j\in J}$ is a $*$-Bessel sequence we define operator $T:\mathcal H\to \ell^2(\mathcal A)$ given by $Tf=\{\langle f, h_j\rangle\}$ then we have
\begin{align*}
\|Tf\|^2\|\sum_{j\in J}\langle f, h_j\rangle\langle h_j, f\rangle\|&\geq \frac{1}{(\sqrt{M}+1)^2}\|\sum_{j\in J}\langle f, f_j\rangle\langle f_j, f\rangle\|\notag\\
&\geq\frac{\|A\|^2}{(\sqrt{M}+1)^2}\|K^*f\|^2.
\end{align*}
This means that $\|Tf\|^2\geq \frac{\|A\|^2}{\mu (\sqrt{M}+1)^2}\|K^*f\|^2$, so by Theorem \ref{adj2} there exists a $\lambda>0$ such that $\sum_{j\in J}\langle f, h_j\rangle\langle h_j, f\rangle\geq \sqrt{\lambda}1_{\mathcal A} \langle K^*f, K^*f\rangle \sqrt{\lambda}1_{\mathcal A}$. Thus by part $(iv)$ of Proposition \ref{propos} $\{h_j\}_{j\in J}$ is a $*$-$L$-frame.

For the converse suppose that $\{h_j\}_{j\in J}$ is a $*$-$L$-frame with the frame bounds $C, D$, respectively and $K$ is co-isometry so $\|K^*f\|=\|f\|$ for any $f\in \mathcal H$. Similarly we have
\begin{align*}
\|\sum_{j\in J}\langle f, f_j-h_j\rangle\langle f_j-h_j, f\rangle\|^{\frac{1}{2}}&\leq\|\sum_{j\in J}\langle f, f_j\rangle\langle f_j, f\rangle\|^{\frac{1}{2}}+\|\sum_{j\in J}\langle f, h_j\rangle\langle h_j, f\rangle\|^{\frac{1}{2}}\\
&\leq \|\sum_{j\in J}\langle f, f_j\rangle\langle f_j, f\rangle\|^{\frac{1}{2}}+\|D\|\|f\|\\
&=\|\sum_{j\in J}\langle f, f_j\rangle\langle f_j, f\rangle\|^{\frac{1}{2}}+\|D\|\|K^*f\|\\
&\leq \|\sum_{j\in J}\langle f, f_j\rangle\langle f_j, f\rangle\|^{\frac{1}{2}}+\frac{\|D\|}{\|A\|}\|\sum_{j\in J}\langle f, f_j\rangle\langle f_j, f\rangle\|^{\frac{1}{2}}\\
&=\left(1+\frac{\|D\|}{\|A\|}\right)\|\sum_{j\in J}\langle f, f_j\rangle\langle f_j, f\rangle\|^{\frac{1}{2}}.
\end{align*}
On the other hands we have
\begin{align*}
\|\sum_{j\in J}\langle f, f_j-h_j\rangle\langle f_j-h_j, f\rangle\|^{\frac{1}{2}}&\leq \|\sum_{j\in J}\langle f, f_j\rangle\langle f_j, f\rangle\|^{\frac{1}{2}}+\|\sum_{j\in J}\langle f, h_j\rangle\langle h_j, f\rangle\|^{\frac{1}{2}}\\
&\leq \|\sum_{j\in J}\langle f, h_j\rangle\langle h_j, f\rangle\|^{\frac{1}{2}}+\|B\|\|f\|\\
&=\|\sum_{j\in J}\langle f, h_j\rangle\langle h_j, f\rangle\|^{\frac{1}{2}}+\|B\|\|K^*f\|,
\end{align*}
Since $R(K)\subseteq R(L)$ and $R(L)$ is closed then by Theorem \ref{adj2} there exists $\lambda>0$ such that $\|K^*f\|^2\leq \lambda \|L^*f\|^2$, $f\in \mathcal H$, so we have
\begin{align*}
\|\sum_{j\in J}\langle f, f_j-h_j\rangle\langle f_j-h_j, f\rangle\|^{\frac{1}{2}}&\leq \|\sum_{j\in J}\langle f, h_j\rangle\langle h_j, f\rangle\|^{\frac{1}{2}}+\|B\|\|K^*f\|\\
&\leq \|\sum_{j\in J}\langle f, h_j\rangle\langle h_j, f\rangle\|^{\frac{1}{2}}+\sqrt{\lambda}\|B\|\|L^*f\|\\
&\leq \|\sum_{j\in J}\langle f, h_j\rangle\langle h_j, f\rangle\|^{\frac{1}{2}}+\frac{\sqrt{\lambda} \|B\|}{\|C\|}\|\sum_{j\in J}\langle f, h_j\rangle\langle h_j, f\rangle\|^{\frac{1}{2}}\\
&=\left(1+\frac{\sqrt{\lambda} \|B\|}{\|C\|}\right)\|\sum_{j\in J}\langle f, h_j\rangle\langle h_j, f\rangle\|^{\frac{1}{2}}.
\end{align*}
Hence with $M=\min\{ (1+\frac{\|D\|}{\|A\|})^2, (1+\frac{\sqrt{\lambda} \|B\|}{\|C\|})^2\}$, \eqref{min} holds.
\end{proof}
\begin{corollary}
Assume that $K, L\in Hom_{\mathcal A}^*(\mathcal H)$ with $R(L)\subseteq R(K)$ and $K$ has closed range. Let $\{f_j\}_{j\in J}$ be a $*$-$K$-frame with $*$-$K$-frame bounds $A, B\in \mathcal Z(\mathcal A)$. If there exists a constant $M>0$, such that for any $f\in \mathcal H$
\begin{align}\label{min}
\|\sum_{j\in J}\langle f, f_j-h_j\rangle\langle f_j-h_j, f\rangle\|\leq M \min\{ \|\sum_{j\in J}\langle f, f_j\rangle\langle f_j, f\rangle\|, \|\sum_{j\in J}\langle f, h_j\rangle\langle h_j, f\rangle\|\}.
\end{align}
Then $\{h_j\}_{j\in J}$ is a $*$-$L$-frame with $*$-frame bounds $mA, nB$, for some positive real numbers $m, n>0$.
\end{corollary}
\begin{proof}
The proof is obvious by Theorem \ref{pertur1} and Lemma \ref{norm2}.
\end{proof}
\begin{corollary}
Assume that $K\in Hom_{\mathcal A}^*(\mathcal H)$ such that $K$ has closed range. Let $\{f_j\}_{j\in J}$ be a $*$-$K$-frame with $*$-$K$-frame bounds $A, B>0$, respectively. If there exists a constant $M>0$, such that
\begin{align*}
\|\sum_{j\in J}\langle f, f_j-h_j\rangle\langle f_j-h_j, f\rangle\|\leq M \min\{ \|\sum_{j\in J}\langle f, f_j\rangle\langle f_j, f\rangle\|, \|\sum_{j\in J}\langle f, h_j\rangle\langle h_j, f\rangle\|\}, \quad f\in \mathcal H
\end{align*}
then $\{h_j\}_{j\in J}$ is a $*$-$K$-frame. The converse is valid for any $K$ is a co-isometry operator.
\end{corollary}
The following  theorem is perturbation result for  $*$-$K$-frames which is a generalization of Theorem 7.3 of \cite{WJ}.
\begin{theorem}\label{pertur2}
Assume that $K, L\in Hom^*_{\mathcal A}(\mathcal H)$ with $R(L)\subseteq R(K)$ with closed range $K$. Let $\{f_j\}_{j\in J}$ be a $*$-$K$-frame, with $*$-$K$-frame bounds $A, B$. If there exists $\alpha, \beta, \gamma\geq0$ such that $\max\{\alpha+\frac{\gamma}{\|A\|}, \beta\}<1$ and
\begin{align}\label{abg}
\|\sum_{j\in J}\langle f, f_j-h_j\rangle\langle f_j-h_j, f\rangle\|^{\frac{1}{2}}\leq  \alpha\|\sum_{j\in J}\langle f, f_j\rangle\langle f_j, f\rangle\|^{\frac{1}{2}}+\beta\|\sum_{j\in J}\langle f, h_j\rangle\langle h_j, f\rangle\|^{\frac{1}{2}}+\gamma \|K^*f\|
\end{align}
Then $\{h_j\}_{j\in J}$ is a $*$-$L$-frame.
\end{theorem}
\begin{proof}
For any $f\in \mathcal H$ we have
\begin{align*}
\|\sum_{j\in J}\langle f, h_j\rangle\langle h_j, f\rangle\|^{\frac{1}{2}}&=\|\{f, h_j\}\|\leq\|\{f, f_j-h_j\}\|+\|\{f, f_j\}\|\\
&= \|\sum_{j\in J}\langle f, f_j-h_j\rangle\langle f_j-h_j, f\rangle\|^{\frac{1}{2}}+\|\sum_{j\in J}\langle f, f_j\rangle\langle f_j, f\rangle\|^{\frac{1}{2}}\\
&\leq (1+\alpha)\|\sum_{j\in J}\langle f, f_j\rangle\langle f_j, f\rangle\|^{\frac{1}{2}}+\beta\|\sum_{j\in J}\langle f, h_j\rangle\langle h_j, f\rangle\|^{\frac{1}{2}}+\gamma \|K^*f\|
\end{align*}
So
\begin{align*}
(1-\beta)\|\sum_{j\in J}\langle f, h_j\rangle\langle h_j, f\rangle\|^{\frac{1}{2}}&\leq (1+\alpha)\|\sum_{j\in J}\langle f, f_j\rangle\langle f_j, f\rangle\|^{\frac{1}{2}}+\gamma \|K^*f\|\\
&\leq ((1+\alpha)+\frac{\gamma}{\|A\|})\|\sum_{j\in J}\langle f, f_j\rangle\langle f_j, f\rangle\|^{\frac{1}{2}}
\end{align*}
Hence
\begin{align}\label{l}
\|\sum_{j\in J}\langle f, h_j\rangle\langle h_j, f\rangle\|^{\frac{1}{2}}&\leq \|B\| (1+\frac{\alpha+\beta+\frac{\gamma}{\|A\|}}{1-\beta})\|f\|.
\end{align}
Similarly 
\begin{align*}
\|\sum_{j\in J}\langle f, h_j\rangle\langle h_j, f\rangle\|^{\frac{1}{2}}\geq (1-\alpha-\frac{\gamma}{\|A\|})\|\sum_{j\in J}\langle f, f_j\rangle\langle f_j, f\rangle\|^{\frac{1}{2}}-\beta \|\sum_{j\in J}\langle f, h_j\rangle\langle h_j, f\rangle\|^{\frac{1}{2}}
\end{align*}
Since $\{h_j\}_{j\in J}$ is a $*$-Bessel sequence the operator $T:\mathcal H\to \ell^2(\mathcal A)$  defined by $Tf=\{\langle f, h_j\rangle\}$ is well-defined. Thus $R(L)\subseteq R(K)$ and Theorem \ref{adj2}imply that there exists $\mu>0$ such that $\mu\|K^*f\|^2\geq \|L^*f\|^2$, $f\in \mathcal H$. So we have
\begin{align}\label{g}
\|Tf\|=\|\sum_{j\in J}\langle f, h_j\rangle\langle h_j, f\rangle\|^{\frac{1}{2}}&\geq (1-\frac{\alpha+\beta+\frac{\gamma}{\|A\|}}{1+\beta})\|\sum_{j\in J}\langle f, f_j\rangle\langle f_j, f\rangle\|^{\frac{1}{2}}\notag\\
&\geq \|A\|(1-\frac{\alpha+\beta+\frac{\gamma}{\|A\|}}{1+\beta})\|K^*f\|\ \ \ f\in \mathcal H.
\end{align}
So by \eqref{l}, \eqref{g}, Theorem \ref{adj2}, Lemma \ref{norm} and part $(iv)$ of Proposition \ref{propos} $\{h_j\}_{j\in J}$ is a $*$-$L$-frame.
\end{proof}
\begin{corollary}
Assume that $K, L\in Hom^*_{\mathcal A}(\mathcal H)$ with $R(L)\subseteq R(K)$ with closed range $K$. Let $\{f_j\}_{j\in J}$ be a $*$-$K$-frame, with $*$-$K$-frame bounds $A, B\in \mathcal Z(\mathcal A)$. If there exists $\alpha, \beta, \gamma\geq0$ such that $\max\{\alpha+\frac{\gamma}{\|A\|}, \beta\}<1$ and
\begin{align}\label{abg}
\|\sum_{j\in J}\langle f, f_j-h_j\rangle\langle f_j-h_j, f\rangle\|^{\frac{1}{2}}\leq  \alpha\|\sum_{j\in J}\langle f, f_j\rangle\langle f_j, f\rangle\|^{\frac{1}{2}}+\beta\|\sum_{j\in J}\langle f, h_j\rangle\langle h_j, f\rangle\|^{\frac{1}{2}}+\gamma \|K^*f\|
\end{align}
Then $\{h_j\}_{j\in J}$ is a $*$-$L$-frame with $*$-frame bounds $mA, nB$, for some positive real numbers $m, n>0$.
\end{corollary}
\begin{proof}
It can be proved by using Theorem \ref{pertur2} and Lemma \ref{norm2}.
\end{proof}
\begin{corollary}
Assume that $K \in Hom^*_{\mathcal A}(\mathcal H)$ with closed range. Let $\{f_j\}_{j\in J}$ be a $*$-$K$-frame, with $*$-$K$-frame bounds $A, B$. If there exists $\alpha, \beta, \gamma\geq0$ such that $\max\{\alpha+\frac{\gamma}{\|A\|}, \beta\}<1$ and
\begin{align}\label{abg}
\|\sum_{j\in J}\langle f, f_j-h_j\rangle\langle f_j-h_j, f\rangle\|^{\frac{1}{2}}\leq  \alpha\|\sum_{j\in J}\langle f, f_j\rangle\langle f_j, f\rangle\|^{\frac{1}{2}}+\|\sum_{j\in J}\langle f, h_j\rangle\langle h_j, f\rangle\|^{\frac{1}{2}}+\gamma \|K^*f\|
\end{align}
Then $\{h_j\}_{j\in J}$ is a $*$-$K$-frame.
\end{corollary}
\begin{corollary}
Let $\{f_j\}_{j\in J}$ be a $*$-frame, with $*$-frame bounds $A, B$. If there exists $\alpha, \beta, \gamma\geq0$ such that $\max\{\alpha+\frac{\gamma}{\|A\|}, \beta\}<1$ and
\begin{align}\label{abg}
\|\sum_{j\in J}\langle f, f_j-h_j\rangle\langle f_j-h_j, f\rangle\|^{\frac{1}{2}}\leq  \alpha\|\sum_{j\in J}\langle f, f_j\rangle\langle f_j, f\rangle\|^{\frac{1}{2}}+\|\sum_{j\in J}\langle f, h_j\rangle\langle h_j, f\rangle\|^{\frac{1}{2}}+\gamma \|K^*f\|
\end{align}
Then $\{h_j\}_{j\in J}$ is a $*$-frame.
\end{corollary}

\end{document}